\documentclass{amsart}
\usepackage{amscd,amsmath,amssymb}
\usepackage{pstricks}
\usepackage{latexsym,amsbsy,mathrsfs}
\usepackage{xy}
\usepackage{xypic}
\usepackage{pgf,tikz}

\newtheorem{thm}{Theorem}[section]
\newtheorem{lem}[thm]{Lemma}
\newtheorem{cor}[thm]{Corollary}
\newtheorem{prop}[thm]{Proposition}

\theoremstyle{definition}

\newtheorem{defn}[thm]{Definition}

\newtheorem{rem}[thm]{Remark}
\newtheorem{rems}[thm]{Remarks}

\numberwithin{equation}{thm}

\begin{document}
\title[RIGHT $n$-ANGULATED CATEGORIES ]
{RIGHT $n$-ANGULATED CATEGORIES ARISING FROM COVARIANTLY FINITE SUBCATEGORIES}

\author{Zengqiang Lin}
\address{ School of Mathematical sciences, Huaqiao University,
Quanzhou\quad 362021,  China.} \email{lzq134@163.com}

\thanks{This work was supported  by  the Natural Science Foundation of Fujian Province (Grants No. 2013J05009)}



\begin{abstract} We define the notion of right $n$-angulated category, which generalizes the notion of right triangulated category. Let $\mathcal{C}$ be an additive category or $n$-angulated category and $\mathcal{X}$ a covariantly finite subcategory, we show that under certain conditions the quotient $\mathcal{C}/\mathcal{X}$ is a right $n$-angulated category. This result generalizes some previous work.
\end{abstract}

\maketitle

\section{Introduction}

Geiss, Keller and Oppermann introduced the notion of $n$-angulated category, which is a $``$higher dimensional" analogue of triangulated category \cite{[GKO]}. Bergh and Thaule introduced a higher $``$octahedral axiom" for an $n$-angulated category and showed that it is equivalent to the mapping cone axiom \cite{[BT1]}. Certain $(n-2)$-cluster tilting subcategories of triangulated categories \cite{[GKO]} and finitely generated free modules over certain local algebras \cite{[BT2]} give rise to $n$-angulated categories. Recently, Jasso introduced $n$-abelian categories, $n$-exact categories and algebraic $n$-angulated category \cite{[J]}. He showed that the quotient category of a Frobenius $n$-exact category has a natural structure of $(n+2)$-angulated category, which generalizes Happel's Theorem \cite[Theorem 2.6]{[Ha]}. $N$-angulated quotient categories induced by mutation pairs  were discussed in \cite{[L]}. Other properties of $n$-angulated categories can see \cite{[BT3]}.

Beligiannis and Marmaridis defined the notion of right triangulated category and showed that if $\mathcal{X}$ is a covariantly finite subcategory of an additive category $\mathcal{C}$, and if any $\mathcal{X}$-monic has a cokernel, then the quotient $\mathcal{C}/\mathcal{X}$ has a structure of right triangulated category. The main aim of this paper is to define the notion of right $n$-angulated category, and discuss the right $n$-angulated categories arising from covariantly finite subcategories. Our two main results, see Theorem \ref{thm} and Theorem \ref{thm2} for details, will generalize some previous work such as \cite[Theorem 2.12]{[BM]}, \cite[Theorem 5.11]{[J]}, \cite[Theorem 3.8]{[L]} and \cite[Theorem 7.2]{[B]}.

This paper is organized as follows. In Section 2, we  define the notion of right $n$-angulated category, recall the definitions of covariantly finite subcategory, $n$-cokernel and $n$-pushout,  and give some preliminaries. In Section 3, we  state and prove our main results, then give some applications.

\section{Definitions and preliminaries}

In this section we define the notion of right $n$-angulated category, then give some other preliminaries.

Let $\mathcal{C}$ be an additive category equipped with an endofunctor $\Sigma:\mathcal{C}\rightarrow\mathcal{C}$ and $n$ an integer greater than or equal to three. An $n$-$\Sigma$-$sequence$ in $\mathcal{C}$ is a sequence of morphisms
$$X_1\xrightarrow{f_1}X_2\xrightarrow{f_2}X_3\xrightarrow{f_3}\cdots\xrightarrow{f_{n-1}}X_n\xrightarrow{f_n}\Sigma X_1.$$
Its {\em left rotation} is the $n$-$\Sigma$-sequence
$$X_2\xrightarrow{f_2}X_3\xrightarrow{f_3}X_4\xrightarrow{f_4}\cdots\xrightarrow{f_{n-1}}X_n\xrightarrow{f_n}\Sigma X_1\xrightarrow{(-1)^n\Sigma f_1}\Sigma X_2.$$  
A $morphism\ of$ $n$-$\Sigma$-$sequences$ is  a sequence of morphisms $\varphi=(\varphi_1,\varphi_2,\varphi_3,\cdots,\varphi_n)$ such that the following diagram commutes
$$\xymatrix{
X_1 \ar[r]^{f_1}\ar[d]^{\varphi_1} & X_2 \ar[r]^{f_2}\ar[d]^{\varphi_2} & X_3 \ar[r]^{f_3}\ar[d]^{\varphi_3} & \cdots \ar[r]^{f_{n-1}}& X_n \ar[r]^{f_n}\ar[d]^{\varphi_n} & \Sigma X_1 \ar[d]^{\Sigma \varphi_1}\\
Y_1 \ar[r]^{g_1} & Y_2 \ar[r]^{g_2} & Y_3 \ar[r]^{g_3} & \cdots \ar[r]^{g_{n-1}} & Y_n \ar[r]^{g_n}& \Sigma Y_1\\
}$$
where each row is an $n$-$\Sigma$-sequence. It is an {\em isomorphism} if $\varphi_1, \varphi_2, \varphi_3, \cdots, \varphi_n$ are all isomorphisms in $\mathcal{C}$.

\begin{defn}
A {\em right} $n$-$angulated\ category$ is a triple $(\mathcal{C}, \Sigma, \Theta)$, where $\mathcal{C}$ is an additive category, $\Sigma$ is an endofunctor of $\mathcal{C}$, and $\Theta$ is a class of $n$-$\Sigma$-sequences (whose elements are called right $n$-angles), which satisfies the following axioms:

(RN1) (a) The class $\Theta$ is closed under isomorphisms, direct sums and direct summands.

(b) For each object $X\in\mathcal{C}$ the trivial sequence
$$X\xrightarrow{1_X}X\rightarrow 0\rightarrow\cdots\rightarrow 0\rightarrow \Sigma X$$
belongs to $\Theta$.

(c) For each morphism $f_1:X_1\rightarrow X_2$ in $\mathcal{C}$, there exists a right $n$-angle whose first morphism is $f_1$.

(RN2) If an $n$-$\Sigma$-sequence belongs to $\Theta$, then its left rotation belongs to $\Theta$.

(RN3) Each commutative diagram
$$\xymatrix{
X_1 \ar[r]^{f_1}\ar[d]^{\varphi_1} & X_2 \ar[r]^{f_2}\ar[d]^{\varphi_2} & X_3 \ar[r]^{f_3}\ar@{-->}[d]^{\varphi_3} & \cdots \ar[r]^{f_{n-1}}& X_n \ar[r]^{f_n}\ar@{-->}[d]^{\varphi_n} & \Sigma X_1 \ar[d]^{\Sigma \varphi_1}\\
Y_1 \ar[r]^{g_1} & Y_2 \ar[r]^{g_2} & Y_3 \ar[r]^{g_3} & \cdots \ar[r]^{g_{n-1}} & Y_n \ar[r]^{g_n}& \Sigma Y_1\\
}$$ with rows in $\Theta$ can be completed to a morphism of  $n$-$\Sigma$-sequences.

(RN4) Given a commutative diagram
$$\xymatrix{
X_1\ar[r]^{f_1}\ar@{=}[d] & X_2 \ar[r]^{f_2}\ar[d]^{\varphi_2} & X_3 \ar[r]^{f_3} & \cdots\ar[r]^{f_{n-2}} & X_{n-1}\ar[r]^{f_{n-1}} & X_n \ar[r]^{f_n} & \Sigma X_1\ar@{=}[d] \\
X_1\ar[r]^{g_1} & Y_2 \ar[r]^{g_2}\ar[d]^{h_2} & Y_3\ar[r]^{g_3} & \cdots\ar[r]^{g_{n-2}} & Y_{n-1}\ar[r]^{g_{n-1}} & Y_n \ar[r]^{g_n} & \Sigma X_1 \\
& Z_3\ar[d]^{h_3} & & & & & \\
& \vdots\ar[d]^{h_{n-2}} & & & & & \\
& Z_{n-1}\ar[d]^{h_{n-1}} & & & & & \\
& Z_n\ar[d]^{h_n} & & & & & \\
& \Sigma X_2 & & & & & \\
}$$ whose top rows and second column belong to $\Theta$. Then there exist morphisms $\varphi_i:X_i\rightarrow Y_i (i=3,4,\cdots,n)$, $\psi_j: Y_j\rightarrow Z_j (j=3,4,\cdots,n)$ and $\phi_k: X_k\rightarrow Z_{k-1} (k=4,5,\cdots,n)$ with the following two properties:

(a) The sequence $(1_{X_1},\varphi_2, \varphi_3,\cdots,\varphi_n)$ is a morphism of $n$-$\Sigma$-sequences.

(b) The $n$-$\Sigma$-sequence
$$X_3\xrightarrow{\left(
                    \begin{smallmatrix}
                      f_3 \\
                      \varphi_3 \\
                    \end{smallmatrix}
                  \right)} X_4\oplus Y_3\xrightarrow{\left(
                             \begin{smallmatrix}
                               -f_4 & 0 \\
                               \varphi_4 & -g_3 \\
                               \phi_4 & \psi_3 \\
                             \end{smallmatrix}
                           \right)}
 X_5\oplus Y_4\oplus Z_3\xrightarrow{\left(
                                       \begin{smallmatrix}
                                         -f_5 & 0 & 0 \\
                                         -\varphi_5 & -g_4 & 0 \\
                                         \phi_5 & \psi_4 & h_3 \\
                                       \end{smallmatrix}
                                     \right)}X_6\oplus Y_5\oplus Z_4$$
$$\xrightarrow{\left(
                                       \begin{smallmatrix}
                                         -f_6 & 0 & 0 \\
                                         \varphi_6 & -g_5 & 0 \\
                                         \phi_6 & \psi_5 & h_4 \\
                                       \end{smallmatrix}
                                     \right)}\cdots\xrightarrow{\scriptsize\left(\begin{smallmatrix}
             -f_{n-1} & 0 & 0 \\
             (-1)^{n-1}\varphi_{n-1} & -g_{n-2} & 0 \\
             \phi_{n-1} & \psi_{n-2} & h_{n-3} \\
             \end{smallmatrix}
             \right)}X_n\oplus Y_{n-1}\oplus Z_{n-2}$$
$$\xrightarrow{\left(
                                                       \begin{smallmatrix}
                                                         (-1)^n\varphi_n &-g_{n-1} &0 \\
                                                          \phi_n& \psi_{n-1}& h_{n-2} \\
                                                       \end{smallmatrix}
                                                     \right)}Y_n\oplus Z_{n-1}\xrightarrow{(\psi_n, h_{n-1})}Z_n\xrightarrow{\Sigma f_2\cdot h_n}\Sigma X_3 \hspace{10mm}$$
belongs to $\Theta$, and $h_n\cdot\psi_n=\Sigma f_1\cdot g_n$.
\end{defn}

\vspace{2mm}

\begin{rems} \label{rem1}
(a) If $n=3$, then the right $n$-angulated category $(\mathcal{C},\Sigma,\Theta)$ is a right triangulated category defined by Beligiannis and Marmaridis \cite{[BM]}.

(b) If $\Sigma$ is an equivalence, and the condition in axiom (RN2) is necessary and sufficient, then the right $n$-angulated category  $(\mathcal{C},\Sigma,\Theta)$ is an $n$-angulated category in the sense of  Geiss-Keller-Oppermann. See \cite[Definition 1.1]{[GKO]} and \cite[Theorem 4.4]{[BT1]}.

(c) We can define left $n$-angulated category dually. If $(\mathcal{C},\Sigma,\Theta)$ is a right $n$-angulated category, $(\mathcal{C},\Omega,\Phi)$ is a left $n$-angulated category, $\Omega$ is a quasi-inverse of $\Sigma$ and $\Theta=\Phi$, then  $(\mathcal{C},\Sigma,\Theta)$ is an $n$-angulated category.
\end{rems}

Let $\mathcal{C}$ be an additive category and $\mathcal{X}$ a subcategory of $\mathcal{C}$. Throughout this paper, when
we say that $\mathcal{X}$ is a subcategory of $\mathcal{C}$, we always mean that $\mathcal{X}$ is full and is closed under
isomorphisms, direct sums and direct summands. We recall that the quotient category $\mathcal{C}/\mathcal{X}$ has the same objects as $\mathcal{C}$ and that the set of morphisms $(\mathcal{C}/\mathcal{X})(A,B)$ is defined as the quotient group $\mathcal{C}(A,B)/[\mathcal{X}](A,B)$, where $[\mathcal{X}](A,B)$ is the set of morphisms from $A$ to $B$ which factor through some object of $\mathcal{X}$. For any morphism $f:A\rightarrow B$ in $\mathcal{C}$, we denote by $\underline{f}$ the image of $f$ under the quotient functor $\mathcal{C}\rightarrow\mathcal{C}/\mathcal{X}$.

A morphism $f:A\rightarrow B$ in $\mathcal{C}$ is called $\mathcal{X}$-$monic$ if $\mathcal{C}(B,X)\xrightarrow{\mathcal{C}(f,X)} \mathcal{C}(A,X)\rightarrow 0$ is exact for any object $X\in\mathcal{X}$. A morphism $f:A\rightarrow X$ in $\mathcal{C}$ is called a $left$ $\mathcal{X}$-$approximation$ $of$ $A$ if $f$ is $\mathcal{X}$-monic and $X\in\mathcal{X}$. The subcategory $\mathcal{X}$ is said to be a $covariantly\ finite\ subcategory$ of $\mathcal{C}$ if any object $A$ of $\mathcal{C}$ has a left $\mathcal{X}$-approximation.  We can defined $\mathcal{X}$-$epic$ morphism, $right$ $\mathcal{X}$-$approximation$ and $contravariantly\ finite\ subcategory$ dually.  The subcategory $\mathcal{X}$ is called {\em functorially finite} if it is both contravariantly finite and covariantly finite.

Let $f: A\rightarrow B$ be a morphism in $\mathcal{C}$. A {\em weak cokernel} of $f$ is a morphism
$g:B\rightarrow C$ such that  $gf=0$ and for each morphism $h: B \rightarrow D$ with $hf = 0$ there exists a (not necessarily unique) morphism $k: C\rightarrow D$ such that $h=kg$. It is easy to see that a weak cokernel $g$ of $f$ is a cokernel of $f$ if and only if $g$ is an epimorphism.

\begin{lem} (\cite[Lemma 2.1]{[J]}) \label{lem1}
Let $A=(A_i,d_i)$ and $B=(B_i,d_i')$ be complexes concentrated in nonnegative degrees, such that for all $i\geq0$ the morphism $d_{i+1}$ is a weak cokernel of $d_i$. If $f:A\rightarrow B$ and $g:A\rightarrow B$ are morphisms of complexes such that $f_0=g_0$, then there exists a homotopy $h$ between $f$ and $g$ such that $h_0=0$.
\end{lem}

\begin{defn} (\cite[Definition 2.2]{[J]})
Let $\mathcal{C}$ be an additive category and $f_0: A_0\rightarrow A_1$ a morphism in $\mathcal{C}$. An $n$-$cokernel\ of$ $f_0$ is a sequence
$$(f_1,f_2,\cdots,f_n): A_1\xrightarrow{f_1}A_2\xrightarrow{f_2}\cdots\xrightarrow{f_n}A_{n+1}$$
such that for all $k=1,2,\cdots,n-1$, the morphism $f_k$ is a weak cokernel of $f_{k-1}$, and $f_{n}$ is a cokernel of $f_{n-1}$.
In this case, we say the sequence
$$A_0\xrightarrow{f_0}A_1\xrightarrow{f_1}A_2\xrightarrow{f_2}\cdots\xrightarrow{f_n}A_{n+1} \ \ \ (2.1)$$
is $right$ $n$-$exact$.
\end{defn}

\begin{rem}
When we say $n$-cokernel we always means that $n$ is a positive integer. We note that the notion of 1-cokernel is  the same as cokernel.
we can define $n$-$kernel$ and {\em left $n$-exact sequence} dually. The sequence (2.1) is called {\em$n$-exact} if it is both right $n$-exact and left $n$-exact.
\end{rem}

\begin{defn}
Let $\mathcal{X}$ be a subcategory of $\mathcal{C}$ and $f_0:A_0\rightarrow A_1$ a morphism in $\mathcal{C}$. We say $f_0$ {\em has a special $n$-cokernel with respect to $\mathcal{X}$}, if $f_0$ has an $n$-cokernel
$$(f_1,f_2,\cdots,f_{n-1},f_n): A_1\xrightarrow{f_1}X_2\xrightarrow{f_2}\cdots\xrightarrow{f_{n-1}} X_n\xrightarrow{f_n}A_{n+1}$$
with $X_i\in\mathcal{X}$.
\end{defn}

\begin{defn} (\cite[Definition 2.11]{[J]})
Let $\mathcal{C}$ be an additive category, $A=A_0\xrightarrow{f_0} A_1\xrightarrow{f_1}\cdots\xrightarrow{f_{n-1}}A_n$ a complex and $g:A_0\rightarrow B_0$ a morphism in $\mathcal{C}$. An $n$-$pushout$ diagram of $A$ along $g$ is a morphism of complexes
$$\xymatrix{
A_0\ar[r]^{f_0}\ar[d]^{g} & A_1\ar[r]^{f_1}\ar[d]^{h_1} & \cdots \ar[r]^{f_{n-2}} & A_{n-1}\ar[r]^{f_{n-1}}\ar[d]^{h_{n-1}}  & A_n \ar[d]^{h_n}\\
B_0\ar[r]^{g_0} & B_1\ar[r]^{g_1}  & \cdots \ar[r]^{g_{n-2}} & B_{n-1}\ar[r]^{g_{n-1}} &  B_n\\
}$$
such that the mapping cone
$$\xymatrixcolsep{2.8pc}
\xymatrix{
A_0\ar[r]^{\left(
                                            \begin{smallmatrix}
                                              -f_0 \\
                                              g \\
                                            \end{smallmatrix}
                                          \right)\ \ \ }
 &  A_1\oplus B_0\ar[r]^{\left(
                                            \begin{smallmatrix}
                                              -f_{1} & 0 \\
                                              h_{1} & g_0 \\
                                            \end{smallmatrix}
                                          \right)}
& A_2\oplus B_1\ar[r]^{\ \ \left(
                                            \begin{smallmatrix}
                                              -f_{2} & 0 \\
                                              h_{2} & g_1 \\
                                            \end{smallmatrix}
                                          \right)}
& \cdots\ar[r]^{\left(
                                            \begin{smallmatrix}
                                              -f_{n-1} & 0 \\
                                              h_{n-1} & g_{n-2} \\
                                            \end{smallmatrix}
                                          \right)\ \ \ }
& A_n\oplus B_{n-1}\ar[r]^{\ \ \ \left(
                                            \begin{smallmatrix}
                                             h_{n} & g_{n-1} \\
                                            \end{smallmatrix}
                                          \right)}
& B_n
}$$
is right $n$-exact.
\end{defn}

The following lemma is useful in the proof of Theorem 3.4.

\begin{lem} \label{lem2}
Let $$\xymatrix{
A_0\ar[r]^{f_0}\ar@{=}[d] & A_1\ar[r]^{f_1}\ar[d]^{h_1} & A_2\ar[r]^{f_2}\ar[d]^{h_2} & \cdots \ar[r]^{f_{n-1}} & A_{n}\ar[r]^{f_{n}}\ar[d]^{h_{n}}  & A_{n+1} \ar[d]^{h_{n+1}}\\
A_0\ar[r]^{g_0} & B_1\ar[r]^{g_1} & B_2\ar[r]^{g_2}  & \cdots \ar[r]^{g_{n-1}} & B_{n}\ar[r]^{g_{n}} &  B_{n+1}\\
}$$ be a commutative diagram of right $n$-exact sequences. Then
$$\xymatrix{
 A_1\ar[r]^{f_1}\ar[d]^{h_1} & A_2\ar[r]^{f_2}\ar[d]^{h_2} & \cdots \ar[r]^{f_{n-1}} & A_{n}\ar[r]^{f_{n}}\ar[d]^{h_{n}}  & A_{n+1} \ar[d]^{h_{n+1}}\\
 B_1\ar[r]^{g_1} & B_2\ar[r]^{g_2}  & \cdots \ar[r]^{g_{n-1}} & B_{n}\ar[r]^{g_{n}} &  B_{n+1}\\
}$$
is an $n$-pushout diagram, in other words,
$$A_1\xrightarrow{d_{0}} A_2\oplus B_1\xrightarrow{d_1} A_3\oplus B_2\xrightarrow{d_2}\cdots\xrightarrow{d_{n-1}}A_{n+1}\oplus B_{n}\xrightarrow{d_{n}}B_{n+1}$$
is a right $n$-exact sequence, where $d_0=\left(
                                            \begin{smallmatrix}
                                              -f_1 \\
                                              h_1 \\
                                            \end{smallmatrix}
                                          \right)
$, $d_i=\left(
                                            \begin{smallmatrix}
                                              -f_{i+1} & 0 \\
                                              h_{i+1} & g_i \\
                                            \end{smallmatrix}
                                          \right)
$ $(i=1,2,\cdots, n-1)$ and $d_n=( h_{n+1} \ g_n)
$.
\end{lem}

\begin{proof}
It is easy to check that $d_{i+1}d_i=0, i=0,1,\cdots, n-1$. Let $s_0=(a_1\ b_1): A_2\oplus B_1\rightarrow C_1$ be a morphism such that $s_0d_0=0$, then $a_1f_1=b_1h_1$. Note that $b_1g_0=b_1h_1f_0=a_1f_1f_0=0$, there exists a morphism $v_1:B_2\rightarrow C_1$ such that $b_1=v_1g_1$ since $g_1$ is a weak cokernel of $g_0$. Now we have $(a_1-v_1h_2)f_1=b_1h_1-v_1g_1h_1=0$, thus there exists a morphism $u_1:A_3\rightarrow C_1$ such that $a_1-v_1h_2=u_1f_2$ since $f_2$ is a weak cokernel of $f_1$. Hence $s_0=(a_1\ b_1)=(-u_1\ v_1)\left(
                           \begin{smallmatrix}
                             -f_2 & 0 \\
                             h_2 & g_1 \\
                           \end{smallmatrix}
                         \right)
$. This shows that $d_1$ is a weak cokernel of $d_0$.

Let $s_{k-1}=(a_k \ b_k): A_{k+1}\oplus B_k\rightarrow C_k$ be a morphism such that $s_{k-1}d_{k-1}=0$ where $k=2,3,\cdots,n-1$. Then $a_kf_k=b_kh_k$ and $b_kg_{k-1}=0$. There exists a morphism $v_k: B_{k+1}\rightarrow C_k$ such that $b_k=v_kg_k$ since $g_k$ is a weak cokernel of $g_{k-1}$.
Note that $(a_k-v_kh_{k+1})f_k=b_kh_k-v_kg_kh_k=0$, there exists a morphism $u_k: A_{k+2}\rightarrow C_k$ such that $a_k-v_kh_{k+1}=u_kf_{k+1}$ since $f_{k+1}$ is a weak cokernel of $f_k$. Hence $s_{k-1}=(a_k \ b_k)=(-u_k \ v_k)\left(
                                                                                \begin{smallmatrix}
                                                                                  -f_{k+1} & 0 \\
                                                                                  h_{k+1} & g_k \\
                                                                                \end{smallmatrix}
                                                                              \right)
$. This shows that $d_k$ is a weak cokernel of $d_{k-1}$.

It remains to show that $d_{n-1}$ is a cokernel of $d_n$. Let $s_{n-1}=(a_n \ b_n): A_{n+1}\oplus B_n\rightarrow C_n$ be a morphism such that $s_{n-1}d_{n-1}=0$. Then $a_nf_n=b_nh_n$ and $b_ng_{n-1}=0$. There exists a morphism $v_n: B_{n+1}\rightarrow C_n$ such that $b_n=v_ng_n$. Now we have $(b_nh_n)f_{n-1}=b_ng_{n-1}h_{n-1}=v_ng_ng_{n-1}h_{n-1}=0$, $a_nf_n=b_nh_n$ and $(v_nh_{n+1})f_n=v_ng_nh_n=b_nh_n$, which implies that $a_n=v_nh_{n+1}$ since $f_n$ is a cokernel of $f_{n-1}$. So $s_{n-1}=(a_n \ b_n)=v_n(h_{n+1} \ g_n)$, and $d_n$ is  a weak cokernel of $d_{n-1}$. Note that $g_n$ is an epimorphism thus so is $d_n=(h_{n+1}, g_n)$.  The proof is completed.
\end{proof}

\section{Main results}

Assume that $\mathcal{C}$ is an additive category and $\mathcal{X}$ a covariantly finite subcategory.

\begin{prop} \label{prop}
If any left $\mathcal{X}$-approximation has a special $n$-cokernel with respect to $\mathcal{X}$, then there is an additive functor $\Sigma:\mathcal{C}/\mathcal{X}\rightarrow\mathcal{C}/\mathcal{X}$.
\end{prop}

\begin{proof}
For any object $A\in\mathcal{C}$, fix a right $n$-exact sequence
$$A\xrightarrow{\alpha_0}X_1\xrightarrow{\alpha_1}X_2\xrightarrow{\alpha_2}\cdots\xrightarrow{\alpha_{n-1}} X_n\xrightarrow{\alpha_n}B$$
where $\alpha_0$ is a left $\mathcal{X}$-approximation of $A$ and $(\alpha_1,\alpha_2,\cdots, \alpha_n)$ is a special $n$-cokernel of $\alpha_0$.
For  any morphism $f:A\rightarrow A'$, since $\alpha_0$ is a left $\mathcal{X}$-approximation, we have the following commutative diagram
$$\xymatrix{
A\ar[r]^{\alpha_0}\ar[d]^{f} & X_1\ar[r]^{\alpha_1}\ar[d]^{x_1} & X_2\ar[r]^{\alpha_2}\ar[d]^{x_2} & \cdots \ar[r]^{\alpha_{n-1}} & X_{n}\ar[r]^{\alpha_{n}}\ar[d]^{x_{n}}  & B \ar[d]^{g}\\
A'\ar[r]^{\alpha_0'} & X'_1\ar[r]^{\alpha_1'} & X'_2\ar[r]^{\alpha_2'} & \cdots \ar[r]^{\alpha_{n-1}'} & X'_{n}\ar[r]^{\alpha_{n}'} &  B'.\\
}$$ Define a functor $\Sigma:\mathcal{C}/\mathcal{X}\rightarrow\mathcal{C}/\mathcal{X}$ such that $\Sigma A=B$ and $\Sigma \underline{f}=\underline{g}$. It is easy to see that $\Sigma$ is a well defined additive functor by Lemma \ref{lem1}.
\end{proof}

\begin{defn}\label{defn}
Let $$A_0\xrightarrow{f_0}A_1\xrightarrow{f_1}A_2\xrightarrow{f_2}\cdots\xrightarrow{f_n}A_{n+1}$$ be a right $n$-exact sequence where $f_0$ is $\mathcal{X}$-monic. Then there exists the following commutative diagram
$$\xymatrix{
A_0 \ar[r]^{f_0}\ar@{=}[d] & A_1 \ar[r]^{f_1}\ar[d]^{a_1} & A_2 \ar[r]^{f_2}\ar[d]^{a_2} & \cdots \ar[r]^{f_{n-1}} & A_{n} \ar[r]^{f_{n}}\ar[d]^{a_{n}}& A_{n+1} \ar[d]^{a_{n+1}\ \ \ \mbox{(3.1)}} \\
A_0 \ar[r]^{\alpha_0} & X_1 \ar[r]^{\alpha_1} & X_2 \ar[r]^{\alpha_2} & \cdots \ar[r]^{\alpha_{n-1}} & X_{n} \ar[r]^{\alpha_{n}} & \Sigma A_0. \\
}$$
The $n$-$\Sigma$-sequence $$A_0\xrightarrow{\underline{f_0}}A_1\xrightarrow{\underline{f_1}}A_2\xrightarrow{\underline{f_2}}\cdots\xrightarrow{\underline{f_{n}}}
A_{n+1}\xrightarrow{(-1)^n\underline{a_{n+1}}} \Sigma A_0$$ is called a $standard\ right\ (n+2)$-$angle$ in $\mathcal{C}/\mathcal{X}$. We define $\Theta$ the class of $n$-$\Sigma$-sequences which are isomorphic to standard right $(n+2)$-angles.
\end{defn}

\begin{rem}
We add the sign $(-1)^n$ in the last morphism of standard right $(n+2)$-angles. In the dual case, maybe we don't need the sign in the first morphism of standard left $(n+2)$-angles. We will see the difference in Corollary \ref{cor}(c). We also can see \cite[Lemma 2.7]{[Ha]} and \cite[Lemma 5.10]{[J]} for the sign. 
\end{rem}

\begin{lem} \label{lem3}
Let $$\xymatrix{
A_0\ar[r]^{f_0}\ar[d]^{h_0} & A_1\ar[r]^{f_1}\ar[d]^{h_1} & A_2\ar[r]^{f_2}\ar[d]^{h_2} & \cdots \ar[r]^{f_{n-1}} & A_{n}\ar[r]^{f_{n}}\ar[d]^{h_{n}}  & A_{n+1} \ar[d]^{h_{n+1}}\\
B_0\ar[r]^{g_0} & B_1\ar[r]^{g_1} & B_2\ar[r]^{g_2}  & \cdots \ar[r]^{g_{n-1}} & B_{n}\ar[r]^{g_{n}} &  B_{n+1}\\
}$$ be a commutative diagram of right $n$-exact sequences, where $f_0$ and $g_0$ are $\mathcal{X}$-monic. Then we have the following commutative diagram of standard right $(n+2)$-angles.
$$\xymatrixcolsep{2.7pc}
\xymatrix{
A_0\ar[r]^{\underline{f_0}}\ar[d]^{\underline{h_0}} & A_1\ar[r]^{\underline{f_1}}\ar[d]^{\underline{h_1}} & A_2\ar[r]^{\underline{f_2}}\ar[d]^{\underline{h_2}} & \cdots \ar[r]^{\underline{f_{n-1}}} & A_{n}\ar[r]^{\underline{f_{n}}}\ar[d]^{\underline{h_{n}}}  & A_{n+1}\ar[r]^{(-1)^n\underline{a_{n+1}}} \ar[d]^{\underline{h_{n+1}}} & \Sigma A_0 \ar[d]^{\Sigma \underline{h_0}}\\
B_0\ar[r]^{\underline{g_0}} & B_1\ar[r]^{\underline{g_1}} & B_2\ar[r]^{\underline{g_2}}  & \cdots \ar[r]^{\underline{g_{n-1}}} & B_{n}\ar[r]^{\underline{g_{n}}} &  B_{n+1}\ar[r]^{(-1)^n\underline{b_{n+1}}} & \Sigma B_0\\
}$$
\end{lem}

\begin{proof} We only need to show that $\Sigma \underline{h_0}\cdot \underline{a_{n+1}}=\underline{b_{n+1}}\cdot\underline{h_{n+1}}$.
By the constructions of standard $(n+2)$-angles and the morphism $\Sigma \underline{h_0}$, we have the following two commutative diagrams
$$\xymatrix{
A_0 \ar[r]^{f_0}\ar@{=}[d] & A_1 \ar[r]^{f_1}\ar[d]^{a_1} & A_2 \ar[r]^{f_2}\ar[d]^{a_2} & \cdots \ar[r]^{f_{n-1}} & A_{n} \ar[r]^{f_{n}}\ar[d]^{a_{n}}& A_{n+1} \ar[d]^{a_{n+1}} \\
A_0 \ar[r]^{\alpha_0} \ar[d]^{h_0} & X_1 \ar[r]^{\alpha_1}\ar[d]^{x_1} & X_2 \ar[r]^{\alpha_2} \ar[d]^{x_2} & \cdots \ar[r]^{\alpha_{n-1}} & X_{n} \ar[r]^{\alpha_{n}} \ar[d]^{x_n} & \Sigma A_0 \ar[d]^{i_0} \\
B_0 \ar[r]^{\beta_0}  & X_1' \ar[r]^{\beta_1} & X_2' \ar[r]^{\beta_2} & \cdots \ar[r]^{\beta_{n-1}} & X'_{n} \ar[r]^{\beta_{n}} & \Sigma B_0, \\
}$$
$$\xymatrix{
A_0 \ar[r]^{f_0}\ar[d]^{h_0} & A_1 \ar[r]^{f_1}\ar[d]^{h_1} & A_2 \ar[r]^{f_2}\ar[d]^{h_2} & \cdots \ar[r]^{f_{n-1}} & A_{n} \ar[r]^{f_{n}}\ar[d]^{h_{n}}& A_{n+1} \ar[d]^{h_{n+1}} \\
B_0 \ar[r]^{g_0}\ar@{=}[d] & B_1 \ar[r]^{g_1}\ar[d]^{b_1} & B_2 \ar[r]^{g_2}\ar[d]^{b_2} & \cdots \ar[r]^{g_{n-1}} & B_{n} \ar[r]^{g_{n}}\ar[d]^{b_{n}}& B_{n+1} \ar[d]^{b_{n+1}} \\
B_0 \ar[r]^{\beta_0} & X_1' \ar[r]^{\beta_1} & X_2' \ar[r]^{\beta_2} & \cdots \ar[r]^{\beta_{n-1}} & X_{n}' \ar[r]^{\beta_{n}} & \Sigma B_0 \\
}$$ where $\Sigma \underline{h_0}=\underline{i_0}$.
Lemma \ref{lem1} implies that $\Sigma \underline{h_0}\cdot \underline{a_{n+1}}=\underline{b_{n+1}}\cdot\underline{h_{n+1}}$.
\end{proof}

Now we can state and prove our first main result.

\begin{thm} \label{thm}
Let $\mathcal{C}$ be an additive category and $\mathcal{X}$ a covariantly finite subcategory. If any $\mathcal{X}$-monic has an $n$-cokernel and any left $\mathcal{X}$-approximation has a special $n$-cokernel with respect to $\mathcal{X}$, then the quotient category $\mathcal{C}/\mathcal{X}$ is a right $(n+2)$-angulated category with respect to the functor $\Sigma$ defined in Proposition \ref{prop} and $(n+2)$-angles defined in Definition \ref{defn}.
\end{thm}

\begin{proof}
It is easy to see that the class of $(n+2)$-angles $\Theta$ is closed under direct sums and direct summands, so (RN1)(a) is satisfied. For any object $A\in\mathcal{C}$, the identity  morphism of $A$ is $\mathcal{X}$-monic. The commutative diagram
$$\xymatrix{
A \ar[r]^{1_A}\ar@{=}[d] & A \ar[r]^{0}\ar[d]^{\alpha_0} & 0 \ar[r]^{0}\ar[d]^{0} & \cdots  \ar[r]^{0} & 0 \ar[r]^{0}\ar[d]^{0} & 0 \ar[d]^{0}\\
A \ar[r]^{\alpha_0} & X_1 \ar[r]^{\alpha_1} & X_2 \ar[r]^{\alpha_2} & \cdots  \ar[r]^{\alpha_{n-1}}& X_n \ar[r]^{\alpha_n} & \Sigma A \\
}$$ shows that $$A\xrightarrow{1_A}A\rightarrow 0\rightarrow\cdots\rightarrow 0\rightarrow \Sigma A$$ belongs to $\Phi$. Thus (RN1)(b) is satisfied.

For any morphism $f:A\rightarrow B$ in $\mathcal{\mathcal{C}}$, the morphism $\left(
                                                          \begin{smallmatrix}
                                                            f \\
                                                            \alpha_0 \\
                                                          \end{smallmatrix}
                                                        \right)
:A\rightarrow B\oplus X_1$ is $\mathcal{X}$-monic, where $\alpha_0$ is a left $\mathcal{X}$-approximation. Let $(f_1,f_2,\cdots,f_n)$ be an $n$-cokernel of $\left(
                                                          \begin{smallmatrix}
                                                            f \\
                                                            \alpha_0 \\
                                                          \end{smallmatrix}
                                                        \right)$,
then we have the following commutative diagram
$$\xymatrix{
A \ar[r]^{\left(
                                                          \begin{smallmatrix}
                                                            f \\
                                                            \alpha_0 \\
                                                          \end{smallmatrix}
                                                        \right)}\ar@{=}[d] & B\oplus X_1 \ar[r]^{f_1}\ar[d]^{a_1} & A_2 \ar[r]^{f_2}\ar[d]^{a_2} & \cdots \ar[r]^{f_{n-1}} & A_{n} \ar[r]^{f_{n}}\ar[d]^{a_{n}}& A_{n+1} \ar[d]^{a_{n+1}} \\
A \ar[r]^{\alpha_0} & X_1 \ar[r]^{\alpha_1} & X_2 \ar[r]^{\alpha_2} & \cdots \ar[r]^{\alpha_{n-1}} & X_{n} \ar[r]^{\alpha_{n}} & \Sigma A. \\
}$$
Thus $$A\xrightarrow{\underline{f}}B\xrightarrow{\underline{f_1}}A_2\xrightarrow{\underline{f_2}}\cdots\xrightarrow{\underline{f_{n}}}
A_{n+1}\xrightarrow{(-1)^n\underline{a_{n+1}}} \Sigma A$$ is a right $(n+2)$-angle. So (RN1)(c) is satisfied.

(RN2) It is easy to see that it suffices to consider the case of standard $(n+2)$-angles. Let $$A_0\xrightarrow{\underline{f_0}}A_1\xrightarrow{\underline{f_1}}A_2\xrightarrow{\underline{f_2}}\cdots\xrightarrow{\underline{f_{n}}}
A_{n+1}\xrightarrow{(-1)^n\underline{a_{n+1}}} \Sigma A_0$$ be a standard $(n+2)$-angle induced by the commutative diagram (3.1).
By Lemma \ref{lem2}, we have a right $n$-exact sequence
$$\xymatrixcolsep{2.7pc}
\xymatrix{
A_1\ar[r]^{\left(
             \begin{smallmatrix}
               -f_1 \\
               a_1 \\
             \end{smallmatrix}
           \right)\ \ \ }
&  A_2\oplus X_1 \ar[r]^{\left(
             \begin{smallmatrix}
               -f_2 & 0\\
               a_2 & \alpha_1 \\
             \end{smallmatrix}
           \right)}
& A_3\oplus X_2 \ar[r]^{\ \ \left(
             \begin{smallmatrix}
               -f_3 & 0\\
               a_3 & \alpha_2 \\
             \end{smallmatrix}
           \right)}
&  \cdots \ar[r]^{\left(
             \begin{smallmatrix}
               -f_n & 0\\
               a_n & \alpha_{n-1} \\
             \end{smallmatrix}
           \right)\ \ \ \ }
& \ A_{n+1}\oplus X_n \ar[r]^{\ \ \ (a_{n+1} \ \alpha_n)}
& \Sigma A_0.
}$$ We claim that the morphism $\left(
                                 \begin{smallmatrix}
                                   -f_1 \\
                                   a_1 \\
                                 \end{smallmatrix}
                               \right): A_1\rightarrow A_2\oplus X_1
$ is $\mathcal{X}$-monic. For any morphism $g:A_1\rightarrow X$ with $X\in\mathcal{X}$, there exists a morphism $h:X_1\rightarrow X$ such that $gf_0 =h\alpha_0$ since $\alpha_0$ is $\mathcal{X}$-monic. Note that $(g-ha_1)f_0=gf_0-h\alpha_0=0$, which implies that there exists a morphism $k:A_2\rightarrow X$ such that $g-ha_1=kf_1$. Thus we have $g=\left(
                                 \begin{smallmatrix}
                                   -k & h \\
                                 \end{smallmatrix}
                               \right)\left(
                                 \begin{smallmatrix}
                                   -f_1 \\
                                   a_1 \\
                                 \end{smallmatrix}
                               \right)$.
The following commutative diagram
$$\xymatrixcolsep{2.8pc}\xymatrix{
A_0 \ar[r]^{\alpha_0}\ar[d]^{f_0} & X_1 \ar[r]^{\alpha_1}\ar[d]^{\left(
                                 \begin{smallmatrix}
                                   0 \\
                                   1 \\
                                 \end{smallmatrix}
                               \right)} & X_2 \ar[r]^{\alpha_2}\ar[d]^{\left(
                                 \begin{smallmatrix}
                                   0 \\
                                   1 \\
                                 \end{smallmatrix}
                               \right)} & \cdots \ar[r]^{\alpha_{n-1}} & X_{n} \ar[r]^{\alpha_{n}}\ar[d]^{\left(
                                 \begin{smallmatrix}
                                   0 \\
                                   1 \\
                                 \end{smallmatrix}
                               \right)}& \Sigma A_{0} \ar@{=}[d] \\
A_1 \ar[r]^{\left(
                                 \begin{smallmatrix}
                                   -f_1 \\
                                   a_1 \\
                                 \end{smallmatrix}
                               \right)}\ar@{=}[d] & A_2\oplus X_1 \ar[r]^{\left(
                                 \begin{smallmatrix}
                                   -f_2 & 0 \\
                                   a_2 & \alpha_1 \\
                                 \end{smallmatrix}
                               \right)}\ar[d]^{b_1} & A_3\oplus X_2 \ar[r]^{\left(
                                 \begin{smallmatrix}
                                   -f_3 & 0 \\
                                   a_3 & \alpha_2 \\
                                 \end{smallmatrix}
                               \right)}\ar[d]^{b_2} & \cdots \ar[r]^{\left(
                                 \begin{smallmatrix}
                                   -f_{n} & 0 \\
                                   a_{n} & \alpha_{n-1} \\
                                 \end{smallmatrix}
                               \right)} & A_{n+1}\oplus X_n \ar[r]^{\left(
                                 \begin{smallmatrix}
                                   a_{n+1} & \alpha_n \\
                                 \end{smallmatrix}
                               \right)}\ar[d]^{b_{n}}& \Sigma A_{0} \ar[d]^{b_{n+1}} \\
A_1 \ar[r]^{\beta_0} & X_1' \ar[r]^{\beta_1} & X_2' \ar[r]^{\beta_2} & \cdots \ar[r]^{\beta_{n-1}} & X_{n}' \ar[r]^{\beta_{n}} & \Sigma A_1 \\
}$$ implies that $\underline{b_{n+1}}=\Sigma \underline{f_0}$ and
$$A_1\xrightarrow{\underline{-f_1}}A_2\xrightarrow{-\underline{f_2}}A_3\xrightarrow{-\underline{f_3}}\cdots\xrightarrow{-\underline{f_{n}}}
A_{n+1}\xrightarrow{\underline{a_{n+1}}} \Sigma A_0\xrightarrow{(-1)^n\Sigma \underline{f_0}} \Sigma A_1 \ \ (3.2)$$
is a right $(n+2)$-angle. Therefore $$A_1\xrightarrow{\underline{f_1}}A_2\xrightarrow{\underline{f_2}}A_3\xrightarrow{\underline{f_3}}\cdots\xrightarrow{\underline{f_{n}}}
A_{n+1}\xrightarrow{(-1)^n\underline{ a_{n+1}}} \Sigma A_0\xrightarrow{(-1)^n\Sigma \underline{f_0}} \Sigma A_1 \ \ (3.3)$$
belongs to $\Theta$ since the $(n+2)$-$\Sigma$-sequences (3.2) and (3.3) are isomorphic.

(RN3) We only consider standard right $(n+2)$-angles.  Suppose that there is a commutative diagram
$$\xymatrixcolsep{2.6pc}\xymatrix{
A_0 \ar[r]^{\underline{f_0}}\ar[d]^{\underline{h_0}} & A_1 \ar[r]^{\underline{f_1}}\ar[d]^{\underline{h_1}} & A_2 \ar[r]^{\underline{f_2}}  & \cdots \ar[r]^{\underline{f_{n}}}& A_{n+1} \ar[r]^{(-1)^n\underline{a_{n+1}}} & \ \  \Sigma A_0 \ar[d]^{\Sigma\underline{h_0} \ \ \ \mbox{(3.4)}}\\
B_0 \ar[r]^{\underline{g_0}} & B_1 \ar[r]^{\underline{g_1}} & B_2 \ar[r]^{\underline{g_2}} & \cdots \ar[r]^{\underline{g_{n}}} & B_{n+1} \ar[r]^{(-1)^n\underline{b_{n+1}}}& \ \ \Sigma B_0\\
}$$
with rows standard right $(n+2)$-angles. Since $\underline{h_{1}}\cdot \underline{f_0}=\underline{g_0}\cdot\underline{h_0}$ holds, $h_1f_0-g_0h_0$ factors through some object $X$ in $\mathcal{X}$. Assume that $h_1f_0-g_0h_0=ba$, where $a:A_0\rightarrow X$ and $b:X\rightarrow B_1$. Since $f_0$ is $\mathcal{X}$-monic, there exists a morphism $c:A_1\rightarrow X$ such that $a=cf_0$. Note that $(h_1-bc)f_0=g_0h_0$, we have the following commutative diagram of right $n$-exact sequences
$$\xymatrix{
A_0 \ar[r]^{f_0}\ar[d]^{h_0} & A_1 \ar[r]^{f_1}\ar[d]^{h_1-bc} & A_2 \ar[r]^{f_2}\ar@{-->}[d]^{h_2} & \cdots \ar[r]^{f_{n}}& A_{n+1} \ar@{-->}[d]^{h_{n+1} \ \ \ \mbox{(3.5)}}\\
B_0 \ar[r]^{g_0} & B_1 \ar[r]^{g_1} & B_2 \ar[r]^{g_2} & \cdots \ar[r]^{g_{n}} & B_{n+1} \\
}$$ by the factorization property of weak cokernels. The diagram (3.4) can be completed to a morphism of right $(n+2)$-angles follows from diagram (3.5) together with Lemma \ref{lem3}.

(RN4) It is enough to consider the case of standard right $(n+2)$-angles. Let
$$A_0\xrightarrow{\underline{f_0}}A_1\xrightarrow{\underline{f_1}}A_2\xrightarrow{\underline{f_2}}\cdots\xrightarrow{\underline{f_{n}}}
A_{n+1}\xrightarrow{(-1)^n\underline{a_{n+1}}} \Sigma A_0$$
$$A_0\xrightarrow{\underline{\varphi_1\cdot f_0}}B_1\xrightarrow{\underline{g_1}}B_2\xrightarrow{\underline{g_2}}\cdots\xrightarrow{\underline{g_{n}}}
B_{n+1}\xrightarrow{(-1)^n\underline{b_{n+1}}} \Sigma A_0$$
$$A_1\xrightarrow{\underline{\varphi_1}}B_1\xrightarrow{\underline{h_1}}C_2\xrightarrow{\underline{h_2}}\cdots\xrightarrow{\underline{h_{n}}}
C_{n+1}\xrightarrow{(-1)^n\underline{c_{n+1}}} \Sigma A_1$$
be three standard right $(n+2)$-angles, where $f_0$ and $\varphi_1$  are $\mathcal{X}$-monic, so that $\varphi_1f_0$ is $\mathcal{X}$-monic too.
By the factorization property of weak cokernels there exist morphisms $\varphi_i: A_i\rightarrow B_i\ (i=2,\cdots, n+1)$, such that we have the following commutative diagram of right $n$-exact sequences
$$\xymatrix{
A_0\ar[r]^{f_0}\ar@{=}[d] & A_1\ar[r]^{f_1}\ar[d]^{\varphi_1} & A_2\ar[r]^{f_2}\ar[d]^{\varphi_2} & \cdots \ar[r]^{f_{n-1}} &  A_{n}\ar[r]^{f_{n}}\ar[d]^{\varphi_{n}}  & A_{n+1} \ar[d]^{\varphi_{n+1} \ \ \ \mbox{(3.6)}}\\
A_0\ar[r]^{\varphi_1f_0} & B_1\ar[r]^{g_1} & B_2\ar[r]^{g_2}  & \cdots \ar[r]^{g_{n-1}} & B_{n}\ar[r]^{g_{n}} &  B_{n+1}.\\
}$$
By Lemma \ref{lem3}, we obtain the following commutative diagram of right $(n+2)$-angles
$$\xymatrixcolsep{2.7pc}
\xymatrix{
A_0\ar[r]^{\underline{f_0}}\ar@{=}[d] & A_1\ar[r]^{\underline{f_1}}\ar[d]^{\underline{\varphi_1}} & A_2\ar[r]^{\underline{f_2}}\ar[d]^{\underline{\varphi_2}} & \cdots \ar[r]^{\underline{f_{n-1}}} & A_{n}\ar[r]^{\underline{f_{n}}}\ar[d]^{\underline{\varphi_{n}}}  & A_{n+1} \ar[r]^{(-1)^n\underline{a_{n+1}}} \ar[d]^{\underline{\varphi_{n+1}}} & \ \ \Sigma A_0 \ar@{=}[d]\\
A_0\ar[r]^{\underline{\varphi_1f_0}} & B_1\ar[r]^{\underline{g_1}} & B_2\ar[r]^{\underline{g_2}}  & \cdots \ar[r]^{\underline{g_{n-1}}} & B_{n}\ar[r]^{\underline{g_{n}}} &  B_{n+1}\ar[r]^{(-1)^n\underline{b_{n+1}}} & \ \ \Sigma A_0.\\
}$$ Diagram (3.6) and Lemma \ref{lem2} implies that
$$\xymatrixcolsep{2.7pc}
\xymatrix{ A_1\ar[r]^{\left(
                    \begin{smallmatrix}
                      -f_1 \\
                      \varphi_1 \\
                    \end{smallmatrix}
                  \right)\ \ }
 & A_2\oplus B_1\ar[r]^{\left(
                    \begin{smallmatrix}
                      -f_2 & 0 \\
                      \varphi_2 & g_1 \\
                    \end{smallmatrix}
                  \right)}
 & A_3\oplus B_2\ar[r]^{\left(
                    \begin{smallmatrix}
                      -f_3 & 0 \\
                      \varphi_3 & g_2 \\
                    \end{smallmatrix}
                  \right)}
 & \cdots\ar[r]^{\left(
                    \begin{smallmatrix}
                      -f_n & 0 \\
                      \varphi_n & g_{n-1} \\
                    \end{smallmatrix}
                  \right)\ \ \ \ }
 &  A_{n+1}\oplus B_n\ar[r]^{\ \ \left(
                    \begin{smallmatrix}
                      \varphi_{n+1} & g_n \\
                    \end{smallmatrix}
                  \right)}
 & B_{n+1}}$$
is a right $n$-exact sequence. There exist morphisms $\phi_i:A_i\rightarrow C_{i-1}$ $(i=3,4,\cdots, n+1)$ and $\psi_j:B_j\rightarrow C_j (j=2,3,\cdots, n+1)$ such that the following diagram is commutative
$$\xymatrix{
A_1\ar[r]^{\left(
                    \begin{smallmatrix}
                      -f_1 \\
                      \varphi_1 \\
                    \end{smallmatrix}
                  \right) \ \
}\ar@{=}[d] & A_2\oplus B_1\ar[r]^{\left(
                    \begin{smallmatrix}
                      -f_2 & 0 \\
                      \varphi_2 & g_1 \\
                    \end{smallmatrix}
                  \right)} \ar@{=}[d] & A_3\oplus B_2 \ar[rr]^{\ \left(
                    \begin{smallmatrix}
                      -f_3 & 0 \\
                      \varphi_3 & g_2 \\
                    \end{smallmatrix}
                  \right)}\ar[d]_{\simeq}^{\left(
                    \begin{smallmatrix}
                      -1 & 0 \\
                      0 & 1 \\
                    \end{smallmatrix}
                  \right)}& & \cdots \ar[rr]^{\left(
                    \begin{smallmatrix}
                      -f_n & 0 \\
                      \varphi_n & g_{n-1} \\
                    \end{smallmatrix}
                  \right)\ \ \ } & & A_{n+1}\oplus B_n\ar[r]^{\ \ \ \left(
                    \begin{smallmatrix}
                      \varphi_{n+1} & g_n \\
                    \end{smallmatrix}
                  \right)}\ar[d]_{\simeq}^{\left(
                    \begin{smallmatrix}
                      (-1)^{n+1} & 0 \\
                      0 & 1 \\
                    \end{smallmatrix}
                  \right)} \ \ \ & B_{n+1} \ar@{=}[d]\\
A_1\ar[r]^{\left(
                    \begin{smallmatrix}
                      -f_1 \\
                      \varphi_1 \\
                    \end{smallmatrix}
                  \right)\ \
}\ar@{=}[d] & A_2\oplus B_1\ar[r]^{\left(
                    \begin{smallmatrix}
                      f_2 & 0 \\
                      \varphi_2 & g_1 \\
                    \end{smallmatrix}
                  \right)} \ar[d]^{(0\ 1)} & A_3\oplus B_2 \ar[rr]^{\ \ \ \left(
                    \begin{smallmatrix}
                      f_3 & 0 \\
                      -\varphi_3 & g_2 \\
                    \end{smallmatrix}
                  \right)}\ar@{-->}[d]^{\left(
                    \begin{smallmatrix}
                      \phi_3 & \psi_2 \\
                    \end{smallmatrix}
                  \right)}& & \cdots\ar[rr]^{\left(
                    \begin{smallmatrix}
                      f_n & 0 \\
                      (-1)^n\varphi_n & g_{n-1} \\
                    \end{smallmatrix}
                  \right)\ \ \ \ } & & A_{n+1}\oplus B_n\ar[r]^{\ \ \ \tiny \left(
                    \begin{smallmatrix}
                      (-1)^{n+1}\varphi_{n+1} & g_n \\
                    \end{smallmatrix}
                  \right)}\ar@{-->}[d]^{ \left(
                    \begin{smallmatrix}
                      \phi_{n+1} & \psi_n \\
                    \end{smallmatrix}
                  \right)} \ \ \ & B_{n+1} \ar@{-->}[d]^{\psi_{n+1}}\\
A_1\ar[r]^{\varphi_1} & B_1 \ar[r]^{h_1} & C_2 \ar[rr]^{h_2}& & \cdots \ar[rr]^{h_{n-1}}& & C_n \ar[r]^{h_n} & C_{n+1}
}$$
where the second row is a right $n$-exact sequence since it is isomorphic to the first row.
By Lemma \ref{lem2} again, we get a right $n$-exact sequence
$$A_2\oplus B_1 \xrightarrow{\left(
                                                                         \begin{smallmatrix}
                                                                              -f_2 & 0 \\
                                                                              -\varphi_2 & -g_1 \\
                                                                              0 & 1\\
                                                                            \end{smallmatrix}
                                                                          \right)  }
  A_3\oplus B_2\oplus B_1\xrightarrow{\left(
                                                                            \begin{smallmatrix}
                                                                              -f_3 & 0 & 0\\
                                                                              \varphi_3 & -g_2 & 0 \\
                                                                              \phi_3 & \psi_2 & h_1\\
                                                                            \end{smallmatrix}
                                                                          \right)}
A_4\oplus B_3\oplus C_2$$
$$\xrightarrow{\left(
                                                                            \begin{smallmatrix}
                                                                              -f_4 & 0  & 0\\
                                                                              -\varphi_4 & -g_3 & 0 \\
                                                                              \phi_4 & \psi_3 & h_2\\
                                                                            \end{smallmatrix}
                                                                          \right)}  \cdots
\xrightarrow{\left(
                                                                            \begin{smallmatrix}
                                                                              -f_n & 0 & 0 \\
                                                                              (-1)^{n+1}\varphi_n & -g_{n-1} & 0 \\
                                                                              \phi_n & \psi_{n-1} & h_{n-2}\\
                                                                            \end{smallmatrix}
                                                                          \right)}
 A_{n+1}\oplus B_n\oplus C_{n-1}$$
 $$\xrightarrow{\left(
                                                                            \begin{smallmatrix}
                                                                            (-1)^{n+2}\varphi_{n+1} & -g_n & 0 \\
                                                                              \phi_{n+1} & \psi_n & h_{n-1}\\
                                                                            \end{smallmatrix}
                                                                          \right)} B_{n+1}\oplus C_{n}\xrightarrow{(\psi_{n+1}\ h_n)} C_{n+1}. \ \ \ (3.7)$$
The following commutative diagram
$$\xymatrixcolsep{6pc}
\xymatrix{ A_2 \ar[r]^{\left(
                                                                         \begin{smallmatrix}
                                                                              f_2 \\
                                                                              \varphi_2 \\
                                                                            \end{smallmatrix}
                                                                          \right) \ \ \ }
                    \ar[d]^{\left(
                    \begin{smallmatrix}
                      -1 \\
                      0  \\
                    \end{smallmatrix}
                    \right)}
&  A_3\oplus B_2\ar[r]^{\left(
                                                                            \begin{smallmatrix}
                                                                              -f_3 & 0 \\
                                                                              \varphi_3 & -g_2 \\
                                                                              \phi_3 & \psi_2\\
                                                                            \end{smallmatrix}
                                                                          \right) \ \ \ } \ar[d]^{\left(
                                                                                                    \begin{smallmatrix}
                                                                                                      1 & 0 \\
                                                                                                      0 & 1 \\
                                                                                                      0 & 0 \\
                                                                                                    \end{smallmatrix}
                                                                                                  \right)}
&  A_4\oplus B_3\oplus C_2 \ar@{=}[d]  \\
A_2\oplus B_1 \ar[r]^{\left(
                        \begin{smallmatrix}
                          -f_2 & 0 \\
                          -\varphi_2 & -g_1 \\
                          0 & 1 \\
                        \end{smallmatrix}
                      \right)} \ar[d]^{\left(
                        \begin{smallmatrix}
                          -1 & 0 \\
                        \end{smallmatrix}
                      \right)}
 & A_3\oplus B_2\oplus B_1 \ar[d]^{\left(
                        \begin{smallmatrix}
                          1 & 0 & 0\\
                          0 & 1 & g_1 \\
                        \end{smallmatrix}
                      \right)}\ar[r]^{\left(
                        \begin{smallmatrix}
                          -f_3 & 0 & 0\\
                          \varphi_3 & -g_2 & 0 \\
                          \phi_3 & \psi_2 & h_1 \\
                        \end{smallmatrix}
                      \right)} & A_4\oplus B_3\oplus C_2 \ar@{=}[d]  \\
 A_2 \ar[r]^{\left(
                                                                         \begin{smallmatrix}
                                                                              f_2 \\
                                                                              \varphi_2 \\
                                                                            \end{smallmatrix}
                                                                          \right) \ \ \ }
&  A_3\oplus B_2\ar[r]^{\left(
                                                                            \begin{smallmatrix}
                                                                              -f_3 & 0 \\
                                                                              \varphi_3 & -g_2 \\
                                                                              \phi_3 & \psi_2\\
                                                                            \end{smallmatrix}
                                                                          \right) \ \ \ }
&  A_4\oplus B_3\oplus C_2  \\
}$$
shows that $$A_2\xrightarrow{\left(
                                                                            \begin{smallmatrix}
                                                                              f_2 \\
                                                                              \varphi_2 \\
                                                                            \end{smallmatrix}
                                                                          \right)} A_3\oplus B_2\xrightarrow{\left(
                                                                            \begin{smallmatrix}
                                                                              -f_3 & 0 \\
                                                                              \varphi_3 & -g_2 \\
                                                                              \phi_3 & \psi_2\\
                                                                            \end{smallmatrix}
                                                                          \right)} A_4\oplus B_3\oplus C_2\xrightarrow{\left(
                                                                            \begin{smallmatrix}
                                                                              -f_4 & 0 & 0 \\
                                                                              -\varphi_4 & -g_3 & 0 \\
                                                                              \phi_4 & \psi_3 & h_3\\
                                                                            \end{smallmatrix}
                                                                          \right)} \cdots\xrightarrow{\left(
                                                                            \begin{smallmatrix}
                                                                              -f_n & 0 & 0 \\
                                                                              (-1)^{n+1}\varphi_n & -g_{n-1} & 0 \\
                                                                              \phi_n & \psi_{n-1} & h_{n-2}\\
                                                                            \end{smallmatrix}
                                                                          \right)}$$
$$ A_{n+1}\oplus B_n\oplus C_{n-1}\xrightarrow{\left(
                                                                            \begin{smallmatrix}
                                                                            (-1)^{n+2}\varphi_{n+1} & -g_n & 0 \\
                                                                              \phi_{n+1} & \psi_n & h_{n-1}\\
                                                                            \end{smallmatrix}
                                                                          \right)} B_{n+1}\oplus C_n\xrightarrow{(\psi_{n+1} \ h_n)} C_{n+1} \ \ \ (3.8)$$
is a direct summand of (3.7), thus it is a right $n$-exact sequence.
We claim that the morphism $\left(
                                                                            \begin{smallmatrix}
                                                                              f_2 \\
                                                                              \varphi_2 \\
                                                                            \end{smallmatrix}
                                                                          \right): A_2\rightarrow A_3\oplus B_2
$ is $\mathcal{X}$-monic. In fact, for any morphism $a:A_2\rightarrow X$ with $X\in\mathcal{X}$, there exists a morphism $b:B_1\rightarrow X$ such that $af_1=b\varphi_1$ since $\varphi_1$ is $\mathcal{X}$-monic. Note that $b(\varphi_1f_0)=af_1f_0=0$, there exists a morphism $c:B_2\rightarrow X$ such that $b=cg_1$. Since $(a-c\varphi_2)f_1=b\varphi_1-cg_1\varphi_1=0$, there exists a morphism $d:A_3\rightarrow X$ such that $a-c\varphi_2=df_2$. Thus $a=(d \ c)\left(
                                                     \begin{smallmatrix}
                                                       f_2 \\
                                                       \varphi_2 \\
                                                     \end{smallmatrix}
                                                   \right)$.
Therefore (3.8) induces a standard right $(n+2)$-angle
$$A_2\xrightarrow{\left(
                                                                            \begin{smallmatrix}
                                                                              \underline{f_2} \\
                                                                              \underline{\varphi_2} \\
                                                                            \end{smallmatrix}
                                                                          \right)} A_3\oplus B_2\xrightarrow{\left(
                                                                            \begin{smallmatrix}
                                                                              -\underline{f_3} & 0 \\
                                                                              \underline{\varphi_3} & -\underline{g_2} \\
                                                                              \underline{\phi_3} & \underline{\psi_2}\\
                                                                            \end{smallmatrix}
                                                                          \right)} A_4\oplus B_3\oplus C_2\xrightarrow{\left(
                                                                            \begin{smallmatrix}
                                                                              -\underline{f_4} & 0 & 0 \\
                                                                              -\underline{\varphi_4} & -\underline{g_3} & 0 \\
                                                                              \underline{\phi_4} & \underline{\psi_3} & \underline{h_3}\\
                                                                            \end{smallmatrix}
                                                                          \right)} \cdots\xrightarrow{\left(
                                                                            \begin{smallmatrix}
                                                                              -\underline{f_n} & 0 & 0 \\
                                                                              (-1)^{n+1}\underline{\varphi_n} & -\underline{g_{n-1}} & 0 \\
                                                                              \underline{\phi_n} & \underline{\psi_{n-1}} & \underline{h_{n-2}}\\
                                                                            \end{smallmatrix}
                                                                          \right)}$$
$$ A_{n+1}\oplus B_n\oplus C_{n-1}\xrightarrow{\left(
                                                                            \begin{smallmatrix}
                                                                            (-1)^{n+2}\underline{\varphi_{n+1}} & -\underline{g_n} & 0 \\
                                                                              \underline{\phi_{n+1}} & \underline{\psi_n} & \underline{h_{n-1}}\\
                                                                            \end{smallmatrix}
                                                                          \right)} B_{n+1}\oplus C_n\xrightarrow{(\underline{\psi_{n+1}} \ \underline{h_n})} C_{n+1}\xrightarrow{(-1)^n\underline{d_{n+1}}} \Sigma A_2.$$
By Lemma \ref{lem3} the following commutative diagram of right $(n+2)$-angles
$$\xymatrixcolsep{4pc}
\xymatrix{
A_0\ar[r]^{\underline{\varphi_1f_0}}\ar[d]^{\underline{f_0}} & B_1\ar[r]^{\underline{g_1}}\ar[d]^{\left(
                                                                                        \begin{smallmatrix}
                                                                                          0 \\
                                                                                          1 \\
                                                                                        \end{smallmatrix}
                                                                                      \right)
} & B_2\ar[r]^{\underline{g_2}}\ar[d]^{\left(
                                                                                        \begin{smallmatrix}
                                                                                          0 \\
                                                                                          1 \\
                                                                                        \end{smallmatrix}
                                                                                      \right)}  & \cdots \\
A_1\ar[r]^{\left(
                    \begin{smallmatrix}
                      -\underline{f_1} \\
                      \underline{\varphi_1} \\
                    \end{smallmatrix}
                  \right)
}\ar@{=}[d] & A_2\oplus B_1\ar[r]^{\left(
                    \begin{smallmatrix}
                      \underline{f_2} & 0 \\
                      \underline{\varphi_2} & \underline{g_1} \\
                    \end{smallmatrix}
                  \right)} \ar[d]^{(0\ 1)} & A_3\oplus B_2 \ar[r]^{\left(
                    \begin{smallmatrix}
                      \underline{f_3} & 0 \\
                      -\underline{\varphi_3} & \underline{g_2} \\
                    \end{smallmatrix}
                  \right)}\ar[d]^{\left(
                    \begin{smallmatrix}
                      \underline{\phi_3} & \underline{\psi_2} \\
                    \end{smallmatrix}
                  \right)} & \cdots \\
A_1\ar[r]^{\underline{\varphi_1}}\ar[d]^{\underline{f_1}} & B_1 \ar[r]^{\underline{h_1}} \ar[d]^{\left(
                    \begin{smallmatrix}
                      0 \\
                       \underline{g_1} \\
                    \end{smallmatrix}
                  \right)} & C_2 \ar[r]^{\underline{h_2}}\ar[d]^{\left(
                    \begin{smallmatrix}
                      0 \\
                      0 \\
                      1\\
                    \end{smallmatrix}
                  \right)} & \cdots \\
A_2 \ar[r]^{\left(
                    \begin{smallmatrix}
                      \underline{f_2} \\
                      \underline{\varphi_2} \\
                    \end{smallmatrix}
                  \right)} & A_3\oplus B_2 \ar[r]^{\left(
                    \begin{smallmatrix}
                      -\underline{f_3} & 0 \\
                     \underline{\varphi_3} & -\underline{g_2}\\
                      \underline{\phi_3} & \underline{\psi_2}\\
                    \end{smallmatrix}
                  \right)} &  A_4\oplus B_3\oplus C_2 \ar[r]^{\ \ \ \ \ \left(
                    \begin{smallmatrix}
                      -\underline{f_4} & 0 & 0 \\
                     \underline{\varphi_4} & -\underline{g_3} & 0\\
                      \underline{\phi_4} & \underline{\psi_3} & \underline{h_2}\\
                    \end{smallmatrix}
                  \right)} & \cdots
}$$
$$\xymatrixcolsep{6pc}
\xymatrix{
\ar[r]^{\underline{g_{n-1}}} & B_{n}\ar[r]^{\underline{g_{n}}}\ar[d]^{\left(
                                                                                        \begin{smallmatrix}
                                                                                          0 \\
                                                                                          1 \\
                                                                                        \end{smallmatrix}
                                                                                      \right)}&  B_{n+1}\ar[r]^{(-1)^n\underline{b_{n+1}}} \ar@{=}[d] & \ \ \Sigma A_0 \ar[d]^{\Sigma \underline{f_0}}\\
\ar[r]^{\left(
                    \begin{smallmatrix}
                      \underline{f_n} & 0 \\
                      (-1)^n\underline{\varphi_n} & \underline{g_{n-1}} \\
                    \end{smallmatrix}
                  \right)}  & A_{n+1}\oplus B_n\ar[r]^{\left(
                    \begin{smallmatrix}
                     (-1)^{n+1}\underline{\varphi_{n+1}} & \underline{g_n} \\
                    \end{smallmatrix}
                  \right)}\ar[d]^{\left(
                    \begin{smallmatrix}
                      \underline{\phi_{n+1}} & \underline{\psi_2} \\
                    \end{smallmatrix}
                  \right)} \ \ \ & B_{n+1} \ar[r]^{(-1)^n\underline{e_{n+1}}}\ar[d]^{\underline{\psi_{n+1}}} & \Sigma A_1 \ar@{=}[d]\\
\ar[r]^{\underline{h_{n-1}}} & C_n \ar[r]^{\underline{h_n}}\ar[d]^{\left(
                    \begin{smallmatrix}
                      0 \\
                      1 \\
                    \end{smallmatrix}
                  \right)} & C_{n+1}\ar[r]^{(-1)^n\underline{c_{n+1}}}\ar@{=}[d] & \Sigma A_1 \ar[d]^{\Sigma\underline{f_1}}\\
\ar[r]^{\left(
                    \begin{smallmatrix}
                     (-1)^{n+2}\underline{\varphi_{n+1}} & -\underline{g_n} & 0\\
                      \underline{\phi_{n+1}} & \underline{\psi_n} & \underline{h_{n-1}}\\
                    \end{smallmatrix}
                  \right)\ \ \ \ \ }  & \ \ \ \ \ B_{n+1}\oplus C_n\ar[r]^{\left(
                    \begin{smallmatrix}
                     \underline{\psi_{n+1}} & \underline{h_n} \\
                    \end{smallmatrix}
                  \right)} & C_{n+1}\ar[r]^{(-1)^n\underline{d_{n+1}}} & \Sigma A_2}$$
shows that $\underline{c_{n+1}}\cdot\underline{\psi_{n+1}}=\Sigma \underline{f_0}\cdot \underline{b_{n+1}}$ and $\underline{d_{n+1}}=\Sigma \underline{f_1}\cdot \underline{c_{n+1}}$. This is what we wanted. We finish the proof.
\end{proof}

\begin{rems}\label{rem2}
(a) If $n=1$, then the condition ``any left $\mathcal{X}$-approximation has a special $n$-cokernel with respect to $\mathcal{X}$" is trivial. Thus Theorem \ref{thm} recover Beligiannis-Marmaridis's result \cite[Theorem 2.12]{[BM]}.

(b) From the proof of Theorem \ref{thm} we see that the condition ``any left $\mathcal{X}$-approximation has a special $n$-cokernel with respect to $\mathcal{X}$" can be weaken as ``for any object $A\in \mathcal{C}$ there exists a left $\mathcal{X}$-approximation $f:A\rightarrow X$ such that $f$ has a special $n$-cokernel with respect to $\mathcal{X}$". The condition ``any $\mathcal{X}$-monic has an $n$-cokernel" is used to prove (RN1)(c), which is redundant in some special case. See Corollary \ref{cor}(a) for details.
\end{rems}

We recall some definitions given by Jasso \cite{[J]}.
Let $\mathcal{A}$ be an $n$-abelian category and $(\mathcal{C},\mathcal{S})$  an $n$-exact subcategory, where $\mathcal{S}$ is a class of admissible  $n$-exact sequences satisfying some axioms similar to exact category. An object $I\in\mathcal{C}$ is called {\em $\mathcal{S}$-injective} if for any admissible monomorphism $f:A\rightarrow B$, the sequence $\mathcal{C}(B,I)\xrightarrow{\mathcal{C}(f,I)}\mathcal{C}(A,I)\rightarrow 0$ is exact. Denote by $\mathcal{I}$ the subcategory of $\mathcal{S}$-injectives. We say that $(\mathcal{C},\mathcal{S})$ {\em has enough $\mathcal{S}$-injectives} if for any object  $A\in\mathcal{C}$, there exists an admissible $n$-exact sequence
$A\rightarrowtail I_1\rightarrow I_2\rightarrow\cdots\rightarrow I_n\twoheadrightarrow B$ with $I_i\in\mathcal{I}$. We can define the notion of {\em $\mathcal{S}$-projective} and {\em having enough $\mathcal{S}$-projectives} dually. Denote by $\mathcal{P}$ the subcategory of $\mathcal{S}$-projectives.
We say that $(\mathcal{C},\mathcal{S})$ is {\em Frobenius} if it has enough $\mathcal{S}$-injectives, has enough $\mathcal{S}$-projectives and if $\mathcal{S}$-injectives and $\mathcal{S}$-projectives coincide. Now we can derive a theorem of Jasso \cite[Theorem 5.11]{[J]}.

\begin{cor} \label{cor}
Let $\mathcal{A}$ be an $n$-abelian category and $(\mathcal{C},\mathcal{S})$  an $n$-exact subcategory.

(a) If $(\mathcal{C},\mathcal{S})$ has enough $\mathcal{S}$-injectives, then the quotient $\mathcal{C}/\mathcal{I}$ is a right $(n+2)$-angulated category.

(b) If $(\mathcal{C},\mathcal{S})$ has enough $\mathcal{S}$-projectives, then the quotient $\mathcal{C}/\mathcal{P}$ is a left $(n+2)$-angulated category.

(c) If $(\mathcal{C},\mathcal{S})$ is Frobenius, then the quotient $\mathcal{C}/\mathcal{I}$ is an $(n+2)$-angulated category.
\end{cor}

\begin{proof}
(a) 
It is easy to see that $\mathcal{I}$ is a covariantly finite subcategory of $\mathcal{C}$. By the definition of having enough $\mathcal{S}$-injectives and Remarks \ref{rem2} (b) we only need to prove (RN1)(c). Let $f:A\rightarrow B$ be a morphism in $\mathcal{C}$, we have the following $n$-pushout diagram  by the axiom of $n$-exact category
$$\xymatrix{
A\ar[r]^{\alpha_0}\ar[d]^{f} & I_1\ar[r]^{\alpha_1}\ar[d]^{a_1} & I_2\ar[r]^{\alpha_2}\ar[d]^{a_2} & \cdots \ar[r]^{\alpha_{n-1}} & I_{n}(\ar[r]^{\alpha_{n}}\ar[d]^{a_{n}}  & \Sigma A) \\
B\ar[r]^{f_0} & A_1\ar[r]^{f_1} & A_2\ar[r]^{f_2}  & \cdots \ar[r]^{f_{n-1}} & A_{n}\\
}$$
Then $$\xymatrixcolsep{2.5pc}
\xymatrix{
A\ar[r]^{\left(
                                            \begin{smallmatrix}
                                              -\alpha_0 \\
                                              f \\
                                            \end{smallmatrix}
                                          \right)}
& I_1\oplus B\ar[r]^{\left(
                                            \begin{smallmatrix}
                                              -\alpha_1 & 0 \\
                                              a_1 & f_0 \\
                                            \end{smallmatrix}
                                          \right)}
& I_2\oplus A_1\ar[r]^{\left(
                                            \begin{smallmatrix}
                                              -\alpha_2 & 0 \\
                                              a_2 & f_1 \\
                                            \end{smallmatrix}
                                          \right)}
& \cdots \ar[r]^{\left(
                                            \begin{smallmatrix}
                                              -\alpha_{n-1} & 0 \\
                                              a_{n-1} & f_{n-2} \\
                                            \end{smallmatrix}
                                          \right)}
& \ \ I_n\oplus A_{n-1}\ar[r]^{\left(
                                            \begin{smallmatrix}
                                             a_n & f_{n-1} \\
                                            \end{smallmatrix}
                                          \right)}
& A_n  \  (3.9)
}$$
is a right $n$-exact sequence. Note that the morphism $\left(
                                            \begin{smallmatrix}
                                              -\alpha_0 \\
                                              f \\
                                            \end{smallmatrix}
                                          \right)$
is $\mathcal{I}$-monic, thus (3.9) induces a right $(n+2)$-angle whose first morphism is $\underline{f}$.

(b) is dual to (a). 
(c) follows from (a) together with Remarks \ref{rem1}(c).
In fact, since $(\mathcal{C},\mathcal{S})$ is Frobenius it is easy to see that the endofunctor $\Sigma: \mathcal{C}/\mathcal{I}\rightarrow\mathcal{C}/\mathcal{I}$ is an equivalence.
We note that a right $n$-exact sequence
 in $\mathcal{C}$  is admissible $n$-exact if and only if the first morphism is $\mathcal{I}$-monic.
Since $(\mathcal{C}, \mathcal{S})$ is Frobenius, we have the following commutative diagram of admissible $n$-exact sequences
$$\xymatrix{
\Sigma^{-1}A_{n+1} \ar[r]^{\beta_0}\ar[d]^{b_{n+1}} & P_1 \ar[r]^{\beta_1}\ar[d]^{b_n} & P_2 \ar[r]^{\beta_2}\ar[d]^{b_{n-1}} & \cdots \ar[r]^{\beta_{n-1}} & P_n \ar[r]^{\beta_n}\ar[d]^{b_1} & A_{n+1} \ar@{=}[d] \\
A_0 \ar[r]^{f_0}\ar@{=}[d] & A_1 \ar[r]^{f_1}\ar[d]^{a_1} & A_2 \ar[r]^{f_2}\ar[d]^{a_2} & \cdots \ar[r]^{f_{n-1}} & A_n \ar[r]^{f_n}\ar[d]_{a_n} & A_{n+1}\ar[d]^{a_{n+1}}\\
A_0 \ar[r]^{\alpha_0} & I_1 \ar[r]^{\alpha_1} &  I_2 \ar[r]^{\alpha_2} & \cdots \ar[r]^{\alpha_{n-1}} & I_n \ar[r]^{\alpha_n} & \Sigma A_0,
}$$  which implies that $$A_0\xrightarrow{\underline{f_0}}A_1\xrightarrow{\underline{f_1}}A_2\xrightarrow{\underline{f_2}}\cdots\xrightarrow{\underline{f_{n}}}
A_{n+1}\xrightarrow{(-1)^n\underline{a_{n+1}}} \Sigma A_0$$ is a standard right $(n+2)$-angle in $\mathcal{C}/\mathcal{I}$ if and only if
$$\Sigma^{-1}A_{n+1}\xrightarrow{\Sigma^{-1}\underline{a_{n+1}}}A_0\xrightarrow{\underline{f_0}}A_1\xrightarrow{\underline{f_1}}
\cdots\xrightarrow{\underline{f_{n-1}}}A_n\xrightarrow{\underline{f_{n}}}
A_{n+1}$$ is a standard left $(n+2)$-angle in $\mathcal{C}/\mathcal{I}$. Thus the class of right $(n+2)$-angles and the class of left $(n+2)$-angles coincide.
\end{proof}

From now on, let $n\geq 3$. In the rest of this section we consider the right $n$-angulated categories arising from covariantly finite subcategories of $n$-angulated categories.

\begin{thm}\label{thm2}
Let $\mathcal{C}$ be an $n$-angulated category and $\mathcal{X}$ a covariantly finite subcategory. If for any object $A\in\mathcal{C}$ there exists an $n$-angle $$A\xrightarrow{\alpha_1}X_1\xrightarrow{\alpha_2}X_2\xrightarrow{\alpha_3}\cdots\xrightarrow{\alpha_{n-2}} X_{n-2}\xrightarrow{\alpha_{n-1}} B\xrightarrow{\alpha_{n}} \Sigma A$$
where $X_i\in\mathcal{X}$ and $\alpha_{1}$ is a left $\mathcal{X}$-approximation. Then the quotient $\mathcal{C}/\mathcal{X}$ is a right $n$-angulated category.
\end{thm}

\begin{proof}
For any object $A\in\mathcal{C}$, fix an $n$-angle $$A\xrightarrow{\alpha_1}X_1\xrightarrow{\alpha_2}X_2\xrightarrow{\alpha_3}\cdots\xrightarrow{\alpha_{n-2}} X_{n-2}\xrightarrow{\alpha_{n-1}} B\xrightarrow{\alpha_{n}} \Sigma A$$
where $X_i\in\mathcal{X}$ and $\alpha_{1}$ is a left $\mathcal{X}$-approximation. For any morphism $f:A\rightarrow A'$, since $\alpha_1$ is a left $\mathcal{X}$-approximation, we have the following commutative diagram
$$\xymatrix{
A\ar[r]^{\alpha_1}\ar[d]^{f} & X_1\ar[r]^{\alpha_2}\ar[d]^{x_1} & X_2\ar[r]^{\alpha_3}\ar[d]^{x_2} & \cdots \ar[r]^{\alpha_{n-2}} & X_{n-2}\ar[r]^{\alpha_{n-1}}\ar[d]^{x_{n-2}}  & B \ar[d]^{g}\ar[r]^{\alpha_n} & \Sigma A \ar[d]^{\Sigma f}\\
A'\ar[r]^{\alpha_1'} & X'_1\ar[r]^{\alpha_2'} & X'_2\ar[r]^{\alpha_3'} & \cdots \ar[r]^{\alpha_{n-2}'} & X'_{n-2}\ar[r]^{\alpha_{n-1}'} &  B' \ar[r]^{\alpha_n'} & \Sigma A'.\\
}$$ Define a functor $T:\mathcal{C}/\mathcal{X}\rightarrow\mathcal{C}/\mathcal{X}$ such that $TA=B$ and $T\underline{f}=\underline{g}$. It is easy to see that $T$ is a well defined additive functor.
Let $$A_1\xrightarrow{f_1}A_2\xrightarrow{f_2}A_3\xrightarrow{f_3}\cdots\xrightarrow{f_{n-1}}A_{n}\xrightarrow{f_n}\Sigma A_1$$ be an $n$-angle where $f_1$ is $\mathcal{X}$-monic. Then there exists the following commutative diagram
$$\xymatrix{
A_1 \ar[r]^{f_1}\ar@{=}[d] & A_2 \ar[r]^{f_2}\ar[d]^{a_2} & A_3 \ar[r]^{f_3}\ar[d]^{a_3} & \cdots \ar[r]^{f_{n-2}} & A_{n-1} \ar[r]^{f_{n-1}}\ar[d]^{a_{n-1}}& A_{n} \ar[d]^{a_{n}}\ar[r]^{f_n} & \Sigma A_1 \ar@{=}[d]  \\
A_1 \ar[r]^{\alpha_1} & X_1 \ar[r]^{\alpha_2} & X_2 \ar[r]^{\alpha_3} & \cdots \ar[r]^{\alpha_{n-2}} & X_{n-2} \ar[r]^{\alpha_{n-1}} & TA_1 \ar[r]^{\alpha_n} & \Sigma A_1. \\
}$$
The $n$-$T$-sequence $$A_1\xrightarrow{\underline{f_1}}A_2\xrightarrow{\underline{f_2}}A_3\xrightarrow{\underline{f_3}}\cdots\xrightarrow{\underline{f_{n-1}}}
A_{n}\xrightarrow{\underline{a_{n}}} TA_1$$ is called a standard right $n$-angle in $\mathcal{C}/\mathcal{X}$. Denote by $\Theta$ the class of $n$-$T$-sequences which are isomorphic to standard right $n$-angles. We can show that $(\mathcal{C}/\mathcal{X}, T, \Theta)$ is a right $n$-angulated category. Since the proof  is similar to \cite[Theorem 3.8]{[L]}, we omit it.
\end{proof}

In particular, if $n=3$, then we get the following well known result.

\begin{cor}
Let $\mathcal{C}$ be a triangulated category and $\mathcal{X}$ a covariantly finite subcategory. Then the quotient $\mathcal{C}/\mathcal{X}$ is a right triangulated category.
\end{cor}

The following corollary  follows immediately from Theorem \ref{thm2} and its dual.

\begin{cor} (cf. \cite[Theorem 3.8]{[L]})
Let $\mathcal{C}$ be an $n$-angulated category and $\mathcal{X}$ a functorially finite  subcategory of $\mathcal{C}$. If $(\mathcal{C},\mathcal{C})$ is an $\mathcal{X}$-mutation pair, that is, for any object $A\in\mathcal{C}$ or $B\in\mathcal{C}$ there exists an $n$-angle $$A\xrightarrow{\alpha_1}X_1\xrightarrow{\alpha_2}X_2\xrightarrow{\alpha_3}\cdots\xrightarrow{\alpha_{n-2}} X_{n-2}\xrightarrow{\alpha_{n-1}} B\xrightarrow{\alpha_{n}} \Sigma A$$
where $X_i\in\mathcal{X}$ , $\alpha_{1}$ is a left $\mathcal{X}$-approximation and $\alpha_{n-1}$ is a right $\mathcal{X}$-approximation.
Then the quotient $\mathcal{C}/\mathcal{X}$ is an $n$-angulated category.
\end{cor}

Before stating another corollary, we give some definitions.
Let $\mathcal{C}$ be an $n$-angulated category. An object $I\in\mathcal{C}$ is called {\em injective} if for any morphism $f:A\rightarrow B$ and any morphism $g:A\rightarrow I$, there exists a morphism $h: B\rightarrow I$ such that $g=hf$. Denote by $\mathcal{I}$ the subcategory of injectives. We say $\mathcal{C}$ {\em has enough injectives} if for any object $A\in\mathcal{C}$, there exists an angle
$$A\rightarrow I_1\rightarrow I_2\rightarrow\cdots\rightarrow I_{n-2}\rightarrow B\rightarrow \Sigma A$$ with $I_i\in\mathcal{I}$.
The notion of {\em having enough projectives} is defined dually. Denote by $\mathcal{P}$ the subcategory of projectives.
If $\mathcal{C}$ has enough injectives, enough projectives and if injectives and projectives coincide, then we say $\mathcal{C}$ is {\em Frobenius}.

The following result is a higher analogue of \cite[Theorem 7.2]{[B]}.

\begin{cor}
Let $\mathcal{C}$ be an $n$-angulated category.

(a) If $\mathcal{C}$ has enough injectives, then the quotient $\mathcal{C}/\mathcal{I}$ is a right $n$-angulated category.

(b) If $\mathcal{C}$ has enough projectives, then the quotient $\mathcal{C}/\mathcal{P}$ is a left $n$-angulated category.

(c) If $\mathcal{C}$ is Frobenius, then the quotient $\mathcal{C}/\mathcal{I}$ is an $n$-angulated category.
\end{cor}


\begin{thebibliography}{9}

\bibitem{[B]} A. Beligiannis. Relative homological algebra and purity in triangulated categories. J. Algebra. 227, no.1, 268-361, 2000.

\bibitem{[BM]} A. Beligiannis,  N. Marmaridis. Left triangulated categories arising from contravariantly
finite subcategories. Commun. Algebra. 22(12), 5021-5036, 1994.

\bibitem{[BT1]} P. A. Bergh, M. Thaule. The axioms for $n$-angulated categories. Algebr. Geom. Topol. 13(4), 2405-2428, 2013.

\bibitem{[BT2]}  P. A. Bergh and M. Thaule. Higher $n$-angulations from local algebras. arXiv:1311.2089,  2013.

\bibitem{[BT3]}  P. A. Bergh and M. Thaule. The Grothendieck group of an $n$-angulated category. J. Pure Appl. Algebra. 218(2), 354-366, 2014.

\bibitem{[GKO]} C. Geiss, B. Keller and S. Oppermann. $n$-angulated categories.  J. Reine Angew. Math.  675, 101-120, 2013.

\bibitem{[Ha]} D. Happel. Triangulated categories in the representation theory of finite dimensional
algebras. London Mathematical Society, LNS 119, Cambridge University Press, Cabridge, 1988.


\bibitem{[J]} G. Jasso. $n$-abelian and $n$-exact categories. arXiv:1405.7805v2, 2014.

\bibitem{[L]} Z. Lin. $n$-angulated quotient categories induced by mutation pairs. arXiv:1409.2716v1, 2014.



\end{thebibliography}
\end{document}